\documentclass{amsart}

\usepackage{amssymb}
\usepackage{hyperref}
\usepackage{tikz}

\usepackage{environ}

\usepackage{import}

\theoremstyle{plain}
\newtheorem{thm}{Theorem}[section]
\newtheorem{theorem}[thm]{Theorem}
\newtheorem*{theorem*}{Theorem}

\newtheorem{corollary}[thm]{Corollary}
\newtheorem{lemma}[thm]{Lemma}
\newtheorem*{lemma*}{Lemma}

\newtheorem{conjecture}[thm]{Conjecture}
\newtheorem*{conjecture*}{Conjecture}

\newtheorem*{fact*}{Fact}

\newtheorem*{question*}{Question}

\newtheorem*{maintheorem*}{Main Theorem}

\theoremstyle{remark}

\theoremstyle{definition}
\newtheorem{definition}[thm]{Definition}

\newtheorem*{standing hypothesis}{Standing Hypothesis}

\newcommand{\thmref}[1]{Theorem~\ref{#1}}

\newcommand{\corref}[1]{Corollary~\ref{#1}}

\newcommand{\lemref}[1]{Lemma~\ref{#1}}
\newcommand{\conjref}[1]{Conjecture~\ref{#1}}

\newcommand{\defref}[1]{Definition~\ref{#1}}

\newcommand{\defn}[1]{\emph{#1}}

\graphicspath{{pictures/}}

\newcommand{\N}{\mathbb{N}}

\newcommand{\R}{\mathbb{R}}

\newcommand{\into}{\hookrightarrow}
\renewcommand{\setminus}{\smallsetminus}

\DeclareMathOperator{\diam}{diam}

\newcommand{\setp}[2]{\left\{#1 \mid #2\right\}}
\newcommand{\abs}[1]{\left| #1 \right|}

\renewcommand{\epsilon}{\varepsilon}

\newcommand{\bdT}{\partial_T}
\DeclareMathOperator{\CAT}{CAT}
\newcommand{\dist}{d}

\title{Extending a rigidity result of Lytchak}

\author{Russell Ricks}


\thanks{This material is based upon work supported by the National Science Foundation under Grant Number NSF 1045119.}

\begin{document}

\maketitle

\begin{abstract}
We extend a rigidity result of Alexander Lytchak by relaxing the requirement that the CAT(1) space be geodesically complete.
\end{abstract}

\section{Introduction}

The following rigidity theorem is due to Alexander Lytchak.

\begin{theorem} [Main Theorem of \cite{lytchak05}]		\label{lytchak rigidity}
Let $Y$ be a complete, geodesically complete, finite-dimensional $\CAT(1)$ space.
If $Y$ contains a proper, closed, involutive subset, then $Y$ is a spherical building or join.
\end{theorem}

One application of this theorem is to the study of complete CAT(0) spaces, the Tits boundaries of which are complete CAT(1) spaces, and finite-dimensional if the CAT(0) space is proper and cocompact%
.
Together with well-known results about $\CAT(0)$ spaces%
, \thmref{lytchak rigidity} gives us the following corollary.

\begin{corollary}						\label{lytchak corollary}
Let $X$ be a proper, cocompact $\CAT(0)$ space.
Assume the Tits boundary, $\bdT X$, of $X$ contains a proper, Tits-closed, involutive subset.
If both $X$ and $\bdT X$ are geodesically complete, then $X$ is a higher rank symmetric space or Euclidean building, or is a nontrivial product.
\end{corollary}

Unfortunately, however, no conditions are known in general that guarantee the Tits boundary of a CAT(0) space is geodesically complete.
We therefore suggest the weaker condition of \defn{essentially $\pi$-geodesically complete}:  Instead of requiring that every geodesic segment be extendible to a geodesic line, we require that every geodesic segment of length less than $\pi$ be approximable by geodesic segments that are extendible to length $\pi$ (see \defref{egc} for a precise statement).

It is known \cite[Theorem C]{kleiner} that the boundary of every proper, cocompact CAT(0) space contains a round sphere (that is, a Euclidean sphere of same dimension as the ambient space).
In this paper, we prove that essential $\pi$-geodesic completeness, combined with the existence of a round sphere, is sufficient to obtain the conclusion of \thmref{lytchak rigidity}.

\begin{theorem} [\thmref{main in body}]				\label{main in intro}
Let $Y$ be a complete $\CAT(1)$ space that contains a round sphere.
Assume $Y$ also contains a proper, closed, involutive subset.
If $Y$ is essentially $\pi$-geodesically complete, then $Y$ is a spherical building or join.
\end{theorem}

As with \corref{lytchak corollary}, we immediately obtain the following corollary.

\begin{corollary}						\label{main cor}
Let $X$ be a proper, cocompact $\CAT(0)$ space.
Assume the Tits boundary, $\bdT X$, of $X$ contains a proper, Tits-closed, involutive subset.
If $X$ is geodesically complete, and $\bdT X$ is essentially $\pi$-geodesically complete, then $X$ is a higher rank symmetric space or Euclidean building, or is a nontrivial product.
\end{corollary}

Naturally, one wonders how often \thmref{main in intro} applies.
We give one setting.

\begin{theorem}	[\thmref{one-dim in body}]			\label{one-dim in intro}
Let $Y$ be a complete $\CAT(1)$ space that contains a round sphere, and assume every pair of points in $Y$ is connected by a geodesic of length at most $\pi$.
If $Y$ has geometric dimension $\dim Y = 1$, then $Y$ is essentially $\pi$-geodesically complete.
\end{theorem}

We remark that the hypothesis that every pair of points be connected by a geodesic of length at most $\pi$ is necessary to prevent the following pathology:
Any $\CAT(1)$ space can be made to fail essential $\pi$-geodesic completeness, simply by attaching a single geodesic arc to the space at one point.
However, this cannot be done without violating the stated hypothesis (as long as there is a round sphere \cite[Lemma 3.1]{bl-centers}).
Moreover, if \thmref{main in intro} holds, then so does this hypothesis.

We believe it likely that \thmref{one-dim in intro} holds in every dimension, provided $Y$ is not a spherical join.
If such a generalization is correct, we suspect one could then extend the results of \cite{ricks-onedim} to every dimension.

\section{Preliminaries}
\label{sec:prelims}

Let $X$ be a metric space.
A \defn{geodesic segment} (or \defn{geodesic}) in $X$ is an isometric embedding $\sigma \colon [0,r] \to X$ for some $r > 0$.
The space $X$ is called \defn{geodesic} if every pair of distinct points $x,y \in X$ can be joined by a geodesic in $X$;
it is called \defn{$\pi$\nobreakdash-geodesic} if every pair of distinct points $x,y \in X$ of distance less than $\pi$ can be joined by a geodesic in $X$;
it is called \defn{geodesically complete} if every geodesic segment in $X$ can be extended to a locally isometric embedding of $\R$ into $X$.

\begin{definition}						\label{egc}
Call a $\pi$-geodesic space $Y$ \defn{essentially $\pi$-geodesically complete} if, for every pair of distinct points $x, y \in Y$ with $d(x,y) < \pi$, and every $\epsilon > 0$, there is a point $z \in Y$ and a geodesic $[x, z] \subset Y$ from $x$ to $z$ of length $\pi$, such that the metric distance $d(y, [x, z])$ from $y$ to $[x, z]$ satisfies $d(y, [x, z]) < \epsilon$.
\end{definition}

A \defn{CAT(1) space} is a $\pi$-geodesic metric space in which distances between points on sides of geodesic triangles of perimeter less than $2\pi$ are never greater than in the corresponding comparison triangles in the Euclidean plane.
We assume the reader is familiar with CAT(0) and CAT(1) spaces (a good reference is \cite{bridson}).

We will write $X * Y$ for the \defn{spherical join} of $\CAT(1)$ spaces $X$ and $Y$.
More precisely, $X * Y$ is the $\CAT(1)$ space characterized by the following three properties:
\begin{enumerate}
\item  $X$ and $Y$ are disjoint, convex subspaces of $X * Y$ (up to isometry).
\item  $d(x,y) = \frac{\pi}{2}$ for all $x \in X$ and $y \in Y$.
\item  For every $z \in X * Y$, there exist $x \in X$, $y \in Y$ such that $d(x,z) + d(z,y) = \frac{\pi}{2}$.
\end{enumerate}

\begin{standing hypothesis}
For the rest of this paper, $Y$ will denote a complete $\CAT(1)$ space.
\end{standing hypothesis}

We will write $\dim Y$ for the \defn{geometric dimension} of $Y$ defined by Kleiner \cite{kleiner}.
Call a subset $K \subset Y$ a \defn{round sphere} if it is isometric to a standard Euclidean sphere of radius one and $\dim K = \dim Y$.
Note that the Tits boundary of every proper, cocompact $\CAT(0)$ space contains a round sphere \cite[Theorem C]{kleiner}.

Call the point $q \in Y$ is an \defn{antipode} of the point $p \in Y$ if $d(p,q) \ge \pi$, and call the set $A \subset Y$ \defn{involutive} if it contains every antipode of every point in $A$.

As in \cite{lytchak05}, an important tool is the ultraproduct%
\footnote{Usually one specifies a basepoint in the construction of the ultraproduct.
We do not need this flexibility, so we simplify by removing this choice from the definition.}
of a $\CAT(1)$ space $Y$.
Fix an \defn{ultrafilter} $\omega$ on $\N$; this is a finitely-additive probability measure on $\N$ such that every subset $A$ of $\N$ is measurable with measure either $0$ or $1$.
We require that $\omega$ be \defn{non-principal}; that is, that $\omega(F) = 0$ for every finite subset $F \subset \N$.
Now given a bounded sequence $(t_n)$ of real numbers, $\omega$ describes a unique \defn{ultralimit} $\lim^\omega t_n$ such that for every $\epsilon > 0$, we have $\abs{t_n - \lim^\omega t_n} < \epsilon$ for $\omega$-almost every $n \in \N$.
Take the set of all sequences $(y_n)$ in $Y$, with the pseudometric $d((p_n),(q_n)) = \lim^\omega d(p_n,q_n)$; the complete metric space $Y^\omega$, formed by taking equivalence classes of distance-zero sequences, is called the \defn{ultraproduct} of $Y$.
For a subset $A \subset Y$, we also define $A^\omega \subset Y^\omega$ to be the closed set $A^\omega = \setp{(a_n) \in Y^\omega}{a_n \in A \text{ for all } n \in \N}$.

Now $\dim Y = \dim Y^\omega$ for any $\CAT(1)$ space $Y$ \cite[Corollary~11.2]{lytchak05}; in particular, if $K$ is a round sphere in $Y$ then so is $K^\omega$ in $Y^\omega$.

There is also a canonical isometric embedding $e \colon Y \into Y^\omega$, mapping each point to the equivalence class of the corresponding constant sequence, $y \in Y \mapsto (y)_n \in Y^\omega$.
Since $Y$ is $\pi$-geodesically complete, $e(Y)$ is $\pi$-convex.

A simple but important observation is the following.

\begin{lemma}							\label{ultralimit g.c.}
Let $Y$ be a complete $\CAT(1)$ space.
If $Y$ is essentially $\pi$-geodesically complete, then $Y^\omega$ is geodesically complete.
\end{lemma}

\begin{proof}
Let $x,y \in Y^\omega$.
Choose representative sequences $(x_n)$ and $(y_n)$ in $Y$ for $x$ and $y$, respectively.
For each $n \in \N$, find a geodesic $[x_n,z_n]$ in $Y$ of length $\pi$ such that $d(y_n, [x_n,z_n]) \le \frac{1}{n}$.
Then the ultralimit $[x,z] \subset Y^\omega$ of the geodesics $[x_n,z_n]$ is also a geodesic of length $\pi$ and satisfies $d(y, [x,z]) = 0$.
\end{proof}

\section{Proof of main theorem}
\label{sec:proof}

\begin{lemma} 							\label{subjoin}
Let $Y$ be a complete $\CAT(1)$ space that splits as a spherical join $Y = A * B$.
Let $S \subset Y$ be a $\pi$-convex subset of $Y$ that contains a round sphere $K$ of $Y$.
Then $S$ splits as a spherical join $S = A' * B'$, where $A' = A \cap S$ and $B' = B \cap S$.
\end{lemma}

\begin{proof}
Let $y \in S \smallsetminus (A' \cup B')$.
There exist $a \in A$ and $b \in B$ such that $y$ belongs to the geodesic segment $[a,b]$ in $Y$, and $\dist(a,y) + \dist(y,b) = \frac{\pi}{2}$.
Since $y \notin A' \cup B'$, both $\dist(a,y), \dist(y,b) > 0$.
By \cite[Lemma 3.1]{bl-centers}, we may find $a' \in K$ such that $\dist(a,a') = \pi$.
Notice that $a' \in A$ because join factors contain all their own antipodes, hence $a' \in A'$.
Thus $\dist(b,a') = \frac{\pi}{2}$.
Hence $\dist(a,a') = \pi = \dist(a,y) + \dist(y,b) + \dist(b,a')$, and thus we have a geodesic segment from $a$ to $a'$ in $Y$ passing through $b$ and $y$.
By $\pi$-convexity, the geodesic segment $[y,a']$ lies completely in $S$; in particular, $b \in B'$.
A symmetric argument shows $a \in A'$.
We conclude that for every $y \in S \smallsetminus (A' \cup B')$ there exist $a \in A'$ and $b \in B'$ such that $\dist(a,y) + \dist(y,b) = \frac{\pi}{2}$.
By \cite[Lemma~11.3]{lytchak05}, $A'$ and $B'$ are nonempty, so our conclusion extends to every $y \in S$.
Thus $S = A' * B'$.
\end{proof}

We recall here two results about spherical buildings.

\begin{lemma} [Corollary 3.5.2 and Proposition 3.10.3 of \cite{kleiner-leeb97}]					\label{subbuilding}
Let $Y$ be a spherical building.
Let $S \subset Y$ be a $\pi$-convex subset of $Y$ that contains a round sphere of $Y$.
Then $S$ is also a spherical building.
\end{lemma}

\begin{lemma} [Lemma 11.5 of \cite{lytchak05}]			\label{lytchak buildings}
Let $Y$ be a geodesically complete $\CAT(1)$ space.
Then $Y$ is a spherical building if and only if $Y^\omega$ is a spherical building.
\end{lemma}

The next lemma is similar to Corollary 11.4 and Lemma 11.5 of \cite{lytchak05}, but the hypotheses are different.

\begin{lemma}							\label{buildings and joins}
Let $Y$ be a complete $\CAT(1)$ space that contains a round sphere.
Then
\begin{enumerate}
\item
$Y$ is a spherical join if and only if $Y^\omega$ is a spherical join.
\item
$Y$ is a spherical building if and only if $Y^\omega$ is a spherical building.
\end{enumerate}
\end{lemma}

\begin{proof}
First, it is straightforward to see that if $Y = A * B$ then $Y^\omega = A^\omega * B^\omega$.
Conversely, if $Y^\omega$ is a spherical join then the canonical embedding $e \colon Y \into Y^\omega$ has $e(Y)$ a $\pi$-convex subset containing a round sphere, because $\dim Y = \dim Y^\omega$; thus $Y$ is a spherical join by \lemref{subjoin}.
This proves item (1).
For (2), first assume $Y$ is a spherical building.
Then $Y$ is geodesically complete, so by \lemref{lytchak buildings}, $Y^\omega$ is a spherical building.
Now assume $Y^\omega$ is a spherical building.
Since $\dim Y = \dim Y^\omega$, by \lemref{subbuilding} we see that $Y$ is a spherical building.
\end{proof}

\begin{lemma} [cf.~ Lemma 11.6 of \cite{lytchak05}]		\label{involutive carry-over}
Let $Y$ be a complete $\CAT(1)$ space.
Assume $A \subset Y$ is a proper, closed, involutive subset of $Y$.
If $Y$ is essentially $\pi$-geodesically complete, then $A^\omega$ is a proper, closed, involutive subset of $Y^\omega$.
\end{lemma}

\begin{proof}
Notice $A^\omega$ is closed for any subset $A \subset Y$.
Next we show $A^\omega$ is proper.
Since $A$ is proper and closed, we may find some $y_0 \in Y$ and some $\epsilon > 0$ such that the open $\epsilon$-ball about $y_0$ is contained in $Y \setminus A$.
Then the open $\epsilon$-ball about $(y_0) \in Y^\omega$ is contained in $Y^\omega \setminus A^\omega$.
Thus $A^\omega$ is a proper subset of $Y^\omega$.

Finally, we show $A^\omega$ is involutive.
Let $x \in A^\omega$, and let $y \in Y^\omega$ be an antipode of $x$ in $Y^\omega$.
Choose representative sequences $(x_n)$ and $(y_n)$ in $Y$ for $x$ and $y$, respectively.
We may assume each $x_n \in A$.
Our goal is to find a sequence $(z_n)$ in $Y$ such that $\omega$-almost every $z_n \in A$, and that $(z_n)$ also represents $y$ in $Y^\omega$.
Now if $d(x_n,y_n) \ge \pi$ for $\omega$-almost every $n$, then $\omega$-almost every $y_n \in A$ already, so assume $d(x_n,y_n) < \pi$ for $\omega$-almost every $n$.
We may then assume $d(x_n,y_n) < \pi$ for every $n$.
So for each $n$, find a geodesic $[x_n,z_n]$ in $Y$ of length $\pi$ such that $d(y_n, [x_n,z_n]) \le \frac{1}{n}$.
Because $\lim^\omega d(x_n,y_n) = \pi$, it follows that $\lim^\omega d(y_n, z_n) = 0$. And since $A$ is involutive, each $z_n \in A$.
Thus $A^\omega$ is involutive.
\end{proof}

\begin{theorem} [\thmref{main in intro}]			\label{main in body}
Let $Y$ be a complete $\CAT(1)$ space that contains a round sphere.
Assume $A \subset Y$ is a proper, closed, involutive subset of $Y$.
If $Y$ is essentially $\pi$-geodesically complete, then $Y$ is a spherical building or join.
\end{theorem}

\begin{proof}
By \lemref{ultralimit g.c.}, $Y^\omega$ is geodesically complete.
By \lemref{involutive carry-over}, $A^\omega$ is a proper, closed, involutive subset of $Y^\omega$.
By \thmref{lytchak rigidity}, $Y^\omega$ is a spherical building or join.
So by \lemref{buildings and joins}, $Y$ is a spherical building or join.
\end{proof}

\section{Essential $\pi$-geodesic completeness}
\label{sec:egc}

It is natural to wonder how often \thmref{main in body} applies.
We give one setting here.

\begin{theorem} [\thmref{one-dim in intro}]			\label{one-dim in body}
Let $Y$ be a complete $\CAT(1)$ space that contains a round sphere, and assume every pair of points in $Y$ is connected by a geodesic of length at most $\pi$.
If $\dim Y = 1$, then $Y$ is essentially $\pi$-geodesically complete.
\end{theorem}

\begin{proof}
Assume $\dim Y = 1$.
Let $[x,y]$ be a geodesic in $Y$ that cannot be extended to a geodesic $[x,z]$ of length $\pi$ in $Y$.
Then for every geodesic $[p,y]$ in $Y$ ending at $y$, the geodesics $[x,y]$ and $[p,y]$ must coincide on a subsegment $[q,y]$ of some positive length.
Now let $K$ be a round sphere in $Y$, and set $\alpha = \frac{1}{3} d(x,y)$ and $A = \setp{p \in K}{d(y,p) \ge \pi - \alpha}$.
By \cite[Lemma~ 3.1]{bl-centers}, $A$ is not empty.
By hypothesis on $Y$, for each $p \in A$ we have a geodesic $[y,p]$ in $Y$ of length $d(y,p) \in [\pi - \alpha, \pi]$.

We claim there is no $q \neq p \in Y$ such that for every $p \in A$, the geodesics $[x,y]$ and $[p,y]$ share the geodesic $[q,y]$ as a subsegment.
For suppose $q$ is such a point.
Let $\delta = d(q,y) > 0$; we may assume $\delta \le \alpha$.
Then every point $p \in K \setminus A$ satisfies $d(p,q) \le d(p,y) + d(y,q) < (\pi - \alpha) + \delta \le \pi$, and meanwhile every point $p \in A \subset K$ satisfies $d(p,q) \le \pi - \delta < \pi$ by hypothesis on $q$.
Hence every point $p \in K$ satisfies $d(p,q) < \pi$, contradicting \cite[Lemma~ 3.1]{bl-centers}.
Therefore, no such $q$ exists.

Now let $\epsilon > 0$, and assume $\epsilon \le \alpha$.
By the previous paragraph, there is some $p \in A$ such that the geodesics $[x,y]$ and $[p,y]$ in $Y$ do not share a subsegment of length $\ge \epsilon$.
But by the first paragraph, $[x,y]$ and $[p,y]$ \emph{do} share a subsegment $[q,y]$ of positive length.
So the concatenated path $[x,q] * [q,p]$ is locally geodesic in $Y$ and length $\ge 2 \alpha + (\pi - \alpha - \epsilon) \ge \pi$.
Thus the initial subsegment $[x,z]$ of $[x,q] * [q,p]$ of length $\pi$ is geodesic, and $d(y, [x, z]) < \epsilon$.
This proves the proposition.
\end{proof}

\begin{corollary}						\label{one-dim cor}
Let $X$ be a proper, cocompact, geodesically complete $\CAT(0)$ space.
Assume the Tits boundary, $\bdT X$, of $X$ contains a proper, Tits-closed, involutive subset.
If $\dim \bdT X = 1$ and $\diam \bdT X = \pi$, then $X$ is a symmetric space or Euclidean building, or is a nontrivial product.
\end{corollary}

We conjecture that \thmref{one-dim in body} is not restricted to dimension one.

\begin{conjecture}						\label{conj:egc}
Let $Y$ be a complete $\CAT(1)$ space that contains a round sphere, and assume every pair of points in $Y$ is connected by a geodesic of length at most $\pi$.
If $Y$ is not essentially $\pi$-geodesically complete, then $Y$ is a spherical join.
\end{conjecture}

We remark that the spherical join of $S^0$ (two points, distance $\pi$ apart) with \emph{any} complete $\CAT(1)$ space containing a round sphere will satisfy all the hypotheses of \conjref{conj:egc}, including the failure of essential $\pi$-geodesic completeness.
Hence the possibility of a spherical join cannot be excluded here, as it is in \thmref{one-dim in body}.

Note also if \conjref{conj:egc} holds, then \corref{one-dim cor} applies in every dimension.

\bibliographystyle{amsplain}
\bibliography{refs}

\end{document}